\documentclass[12pt]{amsart}

\usepackage{amsfonts}
\usepackage[centertags]{amsmath}
\usepackage{amssymb}
\usepackage{amsthm}
\usepackage{newlfont}
\usepackage{graphics}
\usepackage{graphicx}
\usepackage[T1]{fontenc}
\usepackage[all,arc,knot]{xy}
\usepackage{pstricks}
\xyoption{arc}
\usepackage[german]{varioref}
\usepackage{a4,latexsym}
\usepackage{url}

\newtheorem{theorem}{Theorem}
\newtheorem*{theoremA}{Theorem A}
\newtheorem*{theoremB}{Theorem B}
\newtheorem{lemma}{Lemma}[section]

\newtheorem{proposition}[lemma]{Proposition}
\newtheorem{cor}[lemma]{Corollary}

\theoremstyle{definition}
\newtheorem{definition}[lemma]{Definition}
\newtheorem{example}[lemma]{Example}
\newtheorem{remark}[lemma]{Remark}
\newcounter{diagramm}

\newcommand{\RR}{\mathbb{R}}
\newcommand{\NN}{\mathbb{N}}
\newcommand{\HH}{\mathbb{H}}
\newcommand{\ZZ}{\mathbb{Z}}
\newcommand{\CC}{\mathbb{C}}

\newcommand{\bpm}{\begin{pmatrix}}
\newcommand{\epm}{\end{pmatrix}}

\newcommand{\slzwei}[1]{\mbox{SL}_2(#1)}
\newcommand{\emslzwei}[1]{\mbox{\em SL}_2(#1)}
\newcommand{\Gammaquer}{\overline{\Gamma}}
\newcommand{\lcm}{\mbox{lcm}}
\renewcommand{\gcd}{\mbox{gcd}}

\newcommand{\emlcm}{\mbox{\em lcm}}
\newcommand{\mpde}{\mbox{mpde}}
\newcommand{\emmpde}{\mbox{\em mpde}}
\newcommand{\Cr}{Cr}

\newcommand{\omegam}[1]{\mathcal{H}_{#1}}

\begin{document}

\title[The deficiency of being a congruence group]{The deficiency of being a congruence 
group for Veech groups of origamis}

\author{Gabriela Weitze-Schmith\"usen}
\address{Institute for Algebra and Geometry, 
         Karlsruhe Institute of Technology (KIT)}
\email{weitze-schmithuesen@kit.edu}
\urladdr{http://www.math.kit.edu/iag3/~schmithuesen/en}
\thanks{partially supported by the Landesstiftung Baden-W\"urttemberg
within the project {\em With origamis to Teichm\"uller curves in moduli space}}
\date{\today}

\begin{abstract}
We study ``how far away'' a finite index subgroup $\Gamma$
of $\slzwei{\ZZ}$ is from being a congruence group.
For this we define its {\em deficiency of being a congruence 
group}. We show that the index of the image of $\Gamma$ in 
$\slzwei{\ZZ/n\ZZ}$ is the biggest, if $n$ is the general 
Wohlfahrt level. We furthermore show that the Veech groups
of origamis (or square-tiled surfaces) in $\omegam{2}(2)$
are far away from being congruence groups
and that in each genus one finds an infinite family of
origamis such that they are ``as far as possible'' from being 
a congruence group.
\end{abstract}
\maketitle

\section{Introduction}

{\em Teichm\"uller curves} and 
{\em Veech groups of translation surfaces} have
been intensively studied within the last ten years
since they were introduced
in Veech's  famous article \cite{Veech3}.
It was already proven by Veech himself that they
are discrete subgroups of $\slzwei{\RR}$
which are never cocompact. Beside this still only few general
statements are known. A very special class of translation
surfaces are defined by square-tiled surfaces which we also call
{\em origamis}. In this case the translation surface is especially handsome
and the Veech group 
is always a finite index subgroup of $\slzwei{\ZZ}$. The set of 
Teichm\"uller curves coming from origamis is dense in the corresponding
moduli space. This makes origamis an interesting class of translation surfaces
to study.
\\[2mm]
Recall that a subgroup of $\slzwei{\ZZ}$ is called a {\em congruence
group} if it contains one of the principal congruence groups $\Gamma(n)$
(see Section~\ref{section-deficiency}). 
In \cite{Sc1} it was proven that many congruence groups
occur as Veech groups. However, as soon as one fixes a
stratum in the moduli space $\omegam{g}$ of holomorphic unit area
differentials on closed Riemann surfaces of genus $g$, 
it seems that the Veech group of an origami is more likely to be a 
non-congruence group.
It was e.g. shown by Hubert and Leli\`evre in \cite{HL2} that in the
stratum $\omegam{2}(2)$ of holomorphic unit area differentials 
in genus 2  with one zero the Veech groups of
all but one origamis are non-congruence groups.\\

In this article we study ''how far from being a congruence group''
Veech groups of origamis are. For this we define in
Section~\ref{section-deficiency}  for a general 
finite index subgroup $\Gamma$ of $\slzwei{\ZZ}$ 
its {\em deficiency of being a congruence group}. 
As a key ingredient we use the Wohlfahrt level $l$ of $\Gamma$ 
(see Section~\ref{section-deficiency}). If $\Gamma$ 
is a congruence group, this equals the minimal congruence level.
In particular, one has that $\Gamma$ is a congruence group
if and only if the {\em level index} $e_{\Gamma} = [\slzwei{\ZZ/l\ZZ}:p_l(\Gamma)]$
is equal to the index  $d = [\slzwei{\ZZ}:\Gamma]$,
where $p_l\colon\slzwei{\ZZ} \to \slzwei{\ZZ/l\ZZ}$ is the natural
projection. In general the number $e_{\Gamma}$ can be an arbitrary
divisor of $d$ and we define $f_{\Gamma} = \frac{d}{e_{\Gamma}}$
to be the deficiency with respect to $l$. Thus the deficiency
is 1 if and only if $\Gamma$ is a congruence group and we may
take the deficiency as a measure of how far away $\Gamma$ is from being
a congruence group. We say that $\Gamma$ is a 
{\em totally non-congruence group} if its deficiency $f_{\Gamma}$ 
is the index $d$ or equivalently $e_{\Gamma}$ is 1, i.e. the projection $p_l$
is surjective.\\

We can define the numbers $e_{\Gamma} = e_{\Gamma, l}$ and 
$f_{\Gamma} = f_{\Gamma,l}$ in the same way
for arbitrary numbers $l$ possibly different from the Wohlfahrt level. Looking at the case
when $\Gamma$ is a congruence group, one however expects that $e_{\Gamma,l}$
becomes maximal (or equivalently the deficiency $f_{\Gamma,l}$ becomes minimal)
if $l$ is the Wohlfahrt level. We show in Section~\ref{section-deficiency} 
that this is indeed true.

\begin{theorem}[Proof in Section~\ref{section-deficiency}]\label{thm1}
Let $\Gamma$ be a finite index subgroup of $\slzwei{\ZZ}$ and $l \in \NN$.
The deficiency $f_{\Gamma,l}$ becomes minimal if $l$ 
is the  Wohlfahrt level of $\Gamma$.
\end{theorem}

In particular one has the following conclusion.

\begin{cor}
  For a finite index subgroup of $\emslzwei{\ZZ}$ with Wohlfahrt level $l$ we have:
    If $\Gamma$ is a totally non-congruence group, i.e. $e_{\Gamma,l} = 1$,
    then $\Gamma$ surjects to $\emslzwei{\ZZ/n\ZZ}$ for each natural number $n$.
\end{cor}

Furthermore, one should expect that being a congruence group
is an exception and that a general finite index subgroup of 
$\slzwei{\ZZ}$ has a good chance of being a totally non-congruence
group. 
With Theorem~\ref{enough_parabolics} we give a handsome criterion for being 
a totally non-congruence group. For this we consider widths of cusps (see Section~\ref{preps})
at $0$ and
$\infty$ of $\Gamma$ and of a conjugate of $\Gamma$.

\begin{theorem}[Proof in Section~\ref{section-deficiency}]\label{enough_parabolics}
Let $\Gamma_1$ and $\Gamma_2$ be conjugated finite index subgroups
of $\emslzwei{\ZZ}$. Suppose that $\Gamma_1$ has  width $a_1$ at the cusp $0$ and 
width $b_1$ at the cusp $\infty$ and $\Gamma_2$ has width $a_2$ at the cusp $0$ and
width $b_2$ at  the cusp $\infty$.
If 
\begin{equation}\label{relativelyprime}
n_1 = \emlcm(a_1,b_1) \mbox{ and } n_2 = \emlcm(a_2, b_2) \mbox{ are relatively prime},
\end{equation}
then $\Gamma_1$ and $\Gamma_2$ are totally non-congruence groups,
i.e. they surject onto $\emslzwei{\ZZ/n\ZZ}$ for each natural number $n$.
\end{theorem}

We then study the deficiency of the Veech groups of origamis
in $\omegam{2}(2)$. Using the criterion from Theorem~\ref{enough_parabolics} we 
show that it is always $d$ or $\frac{d}{3}$, 
i.e. they are totally non-congruence groups or at least far
away from being congruence groups.

\begin{theorem}\label{stratumlab}
Let $O$ be an origami in $\omegam{2}(2)$ with $j$ squares and let $\Gamma(O)$
be its Veech group. We distinguish the two different cases that $O$
is in the orbit $A_j$ or $B_j$ in the classification
of orbits in $\omegam{2}(2)$ by McMullen and Hubert/Leli\`evre  (see Section~\ref{hone}).
\begin{itemize}
\item[i)]
If $j$ is odd, $j \geq 5$ and $O$ is in $B_j$, then $\Gamma(O)$ is a totally 
non-congruence group, i.e. it surjects onto $\emslzwei{\ZZ/n\ZZ}$
for each $n \in \NN$.        
\item[ii)] 
If $j$ is even, or $j$ is odd and $O$ is in $A_j$, or $j = 3$, then 
the deficiency $f_{\Gamma,l}$ with respect to the Wohlfahrt level $l$
of $\Gamma(O)$ 
is equal to $\frac{d}{3}$. I.e. the index of its image in $\emslzwei{\ZZ/l\ZZ}$
is 3. 
\end{itemize}
\end{theorem}

\begin{cor}
The Veech groups of all origamis in $\omegam{2}(2)$ 
form {\em an honest family of non-congruence groups}, i.e. an infinite family such
that the level index $e_l$ is bounded by a constant. More precisely 
their level index is bounded by 3.
\end{cor}

{\em Remark:}
A similar approach as in Theorem~\ref{stratumlab} 
is independently used by Christian Wei\ss\
in \cite{weiss} in the proof of Theorem 5.28, 
where he calculates the Euler characteristic of
twisted Teichm\"uller curves arising from the non-origami
$L$-shaped translation surfaces 
$L_D$ as defined in \cite{McM} and \cite{C}.
In this case the Veech group is a subgroup of $\slzwei{\mathcal{O}_D}$,
where $\mathcal{O}_D$ is a real quadratic order of discriminant $D$
and one can similarly look at groups defined by congruence conditions 
in $\slzwei{\mathcal{O}_D}$.\\

The criterion from Theorem~\ref{enough_parabolics}
is quite general and finally allows us to detect 
infinite families of origamis whose Veech groups
are all totally non-congruence groups  also in higher
genus.

\begin{theorem}\label{onezero}
For each $g \geq 3$, the stratum $\omegam{g}(2g-2)$
contains an infinite family of origamis whose Veech groups
are totally non-congruence groups.
\end{theorem}

{\bf Acknowledgements:}
I am indebted to Stefan K\"uhnlein,
who has drawn my attention to non congruence groups
and from whom I have learned the result of Wohlfahrt. 
Thanks to Matthias Nagel, who detected
that the level index of the Veech group of the $L$-origamis
is always 1 or 3 while doing computer experiments within
the project {\em With origamis to Teichm\"uller curves in moduli space}.
Furthermore I am very grateful to
Carlos Matheus for discussions which in the end lead to this article,
and to Matt Bainbridge for conversations which
always brought the project forward.
Furthermore, I would like to thank Frank Herrlich, Myriam Finster 
and Christian Wei\ss\ for proof reading of earlier versions
of this article.\\
I would like to express my gratitude to the
Landesstiftung Baden-W\"urttemberg. The project   
{\em With origamis to Teichm\"uller curves in moduli space}
sponsored by the Landesstiftung Baden-W\"urttemberg
lead to the starting point of this article.
Finally, I am grateful to the Hausdorff Research Institute
for Mathematics in Bonn as well as to the organisers of the
trimester program {\em Geometry and Dynamics of Teichm\"uller 
spaces}, during which part of this work was done. 


\section{Preliminaries}\label{preps}

There is today a wide literature on the topic of translation surfaces,
Veech groups and Teichm\"uller curves. We summarise in the following
paragraphs the basics that we need in this article.
The reader can find more detailed introductions 
to these topics with more hints  to literature
e.g. in \cite{Zorich},  \cite{HeSc}, \cite{M} or \cite{K}.

\subsection{Translation surfaces, origamis and Veech groups}
Recall that a {\em translation surface} is a surface $X^*$ together 
with an atlas $\mu$ such that all transition maps are translations.
A translation atlas is in particular holomorphic
and thus $(X^*,\mu)$ has the structure of a Riemann surface. We
call $(X^*,\mu)$ {\em precompact of genus $g$}, 
if $X^* = X\backslash\{P_1, \ldots, P_k\}$, where $X$ is a closed 
Riemann surface of  genus $g$ and $P_1$, \ldots, $P_k$ 
are finitely many cone points on $X$.
There is a well-known bijection between precompact translation surfaces of genus $g$
(up to translations) and pairs $(X,\omega)$ (up to isomorphisms preserving $\omega$) 
where $X$ is a closed Riemann surface of genus $g$ and $\omega$ is a holomorphic 
differential on $X$ (see e.g. \cite[Section 1.1.3]{HS1}). 
We denote by $\omegam{g}(a_1, \ldots, a_r)$
the moduli space of unit area holomorphic differentials 
on a genus $g \geq 1$ Riemann surface with $r$
zeroes of order $a_1$, \ldots, $a_r$. \\
One of the simplest ways to explicitly construct translation surfaces 
goes as follows: take
finitely many copies of the unit square. Glue them along their edges
via translations such that each left edge is glued to precisely on right edge,
each upper edge to precisely one lower edge and the resulting surface
is connected. This way we obtain a closed surface $X$ with a set $\tilde{S}$ of finitely many 
marked points which come from the vertices of the squares. The Euclidean structure of the 
plane defines a translation structure on $X^* = X\backslash\tilde{S}$. The points in $\tilde{S}$ become
cone-type singularities. We call $X$ a {\em square-tiled surface} or - emphasising its
combinatorial structure - an {\em origami}. 
If the angle around a point $p$ in $\tilde{S}$ is $2\pi$,
we can extend the translation structure into this point. Points in $\tilde{S}$
with total angle bigger than $2\pi$ are called {\em singularities}. Straight lines with respect
to the translation structure which connect two singularities are called {\em saddle connections}.
To each saddle connection we associate its developing vector $\vec{v}$ in the Euclidean plane.
In this article we will only work with {\em primitive origamis}, 
which means that the developing vectors of the saddle connections span $\ZZ^2$ 
(see \cite[Lemma 2.1 and paragraph below]{HL}).\\
There are several equivalent ways to 
combinatorially describe an origami, see \cite{Sc2}. 
A handsome one, if we want to notate explicit examples, 
is the
description by permutations given as follows. Label the squares by $Q_1$, \ldots, $Q_d$.
Let $\sigma_a$  be the permutation in $S_d$ which assigns $i$ the number of the right 
neighbour of the square $Q_i$ and $\sigma_b$ the permutation which assigns $i$ the number
of the upper neighbour of $Q_i$. The pair $(\sigma_a,\sigma_b)$
completely describes the surfaces up to isomorphism. 
Renumbering the squares leads to simultaneous conjugation. Since $X$ is connected,
the group generated by  $\sigma_a$ and $\sigma_b$ acts transitively on $\{1,\ldots, d\}$.
Conversely each transitive pair $(\sigma_a,\sigma_b)$ of permutations in $S_d$ 
defines an origami.\\
A natural description of origamis  is via coverings of the torus as 
described in the following.
The square tiling  naturally defines a finite covering to the torus $E$
formed by gluing opposite edges of one square. This
covering is ramified over at most one point, namely the one point
on $E$ arising from the vertices of the one square. 
This leads
to the following equivalent definition. An {\em origami} is the isomorphism class of a finite connected covering 
$p\colon X \to E$ ramified at most over one point $\infty$. Here two
coverings $p_1\colon X_1 \to E$ and $p_2\colon X_2 \to E$ 
are said to be {\em isomorphic} if there is a map $f\colon X_1 \to X_2$ 
with $p_2\circ f = p_1$.\\
The point of view of coverings naturally leads us to monodromy.
If we are given an origami $p\colon X \to E$, we may remove the point $\infty$ from $E$ and all 
of its preimages in $X$ and obtain an unramified covering of the once-punctured
torus $E^*$. The fundamental group of $E^*$ is $F_2 = F_2(x,y)$,
the free group in two generators $x$ and $y$. We choose $x$ and $y$
to be the simple closed horizontal and vertical curve on $E^*$, respectively. 
By the classical theory of coverings this 
defines a monodromy action $\rho\colon F_2 \to S_d$, where $d$ is the 
degree of the covering $p$ if we choose a base point: The fiber
of $p$ over the base point is a set with $d$ elements. $F_2$ acts on it from
the right, namely for $w\in F_2$ and for $\gamma_w$ the corresponding
closed curve on $E^*$, $\rho(w)$ maps a point $q$ in the fiber to
the endpoint of the lift of $\gamma_w$ with starting point $q$. This
depends on the choise of the base point only up to conjugation.
Recall that a {\em monodromy action} is an
anti-homomorphism such that its image acts transitively on the set $\{1, \ldots, d\}$.
Observe that we have for the pair of permutations $(\sigma_a,\sigma_b)$
from above that $\sigma_a = \rho(x)$ and $\sigma_b = \rho(y)$.\\

\begin{figure}[h]
\begin{center}
  \setlength{\unitlength}{1cm}
  \begin{picture}(6,5)
    \put(0,2){\framebox(1,1){1}}
    \put(0,3){\framebox(1,1){2}}
    \put(0,4){\framebox(1,1){3}}
    \put(1,2,){\framebox(1,1){4}}
    \put(1,0){\framebox(1,1){5}}
    \put(1,1){\framebox(1,1){6}}
    \put(2,2,){\framebox(1,1){7}}
    \put(3,2){\framebox(1,1){8}}
    \put(4,2){\framebox(1,1){9}}
    \put(0,2){\circle*{.2}}
    \put(0,3){\circle*{.2}}
    \put(0,5){\circle*{.2}}
    \put(1,2){\circle*{.2}}
    \put(1,3){\circle*{.2}}
    \put(1,5){\circle*{.2}}
    \put(1,0){\circle*{.2}}
    \put(2,0){\circle*{.2}}
    \put(2,2){\circle*{.2}}
    \put(2,3){\circle*{.2}}
    \put(5,2){\circle*{.2}}
    \put(5,3){\circle*{.2}}
  \end{picture}
\end{center}
\begin{center}
\refstepcounter{diagramm}{\it Figure \arabic{diagramm}}:\label{ori-example}
{\em An origami in $\Omega_3(4)$: Opposite edges are glued.}
\end{center}
\end{figure}

\begin{example}
Figure~\ref{ori-example} shows an origami with 9 squares. 
Opposite edges are glued. The origami is given by the
two permutations:
\[ \sigma_a = (1,4,7,8,9) \mbox{ and } \sigma_b = (1,2,3)(4,5,6).\]
The marked vertices glue to one point 
on the surface $X$ of total angle $10\pi$. 
The unmarked vertices become regular points, i.e. the total 
angle around them is $2\pi$.
A simple Euler characteristic calculation
shows that the genus of $X$ is 3. 
\end{example}

In \cite{Veech3} Veech introduced a group associated to a translation surface $(X^*,\mu)$
which is today called the {\em Veech group}. Let Aff$^+(X^*,\mu)$ be the
group of orientation preserving affine homeomorphisms of $X^*$, i.e. those homeomorphisms
which are affine on charts. Thus locally they are of the form
${x\choose y} \mapsto A \cdot {x\choose y} + {a\choose b}$
with $A \in \slzwei{\RR}$ and ${a\choose b} \in \RR^2$. The matrix $A$ is
independent of the chart, since the transition maps are translations.
We obtain a homomorphism $D\colon \mbox{Aff}^+(X,\mu) \to \slzwei{\RR}, f \mapsto A$
called the {\em derivative map}.
The {\em Veech group $\Gamma(X^*,\mu)$} is the image of $D$. Veech already proved 
in \cite{Veech3} that $\Gamma(X^*,\mu)$ is a discrete subgroup of $\slzwei{\RR}$
and that it is never cocompact. One reason why Veech groups were intensively studied
in the last fifteen years is that they are closely related to 
{\em Teichm\"uller curves} (see e.g. \cite{McM}, \cite{HeSc}). Namely, each translation surface defines a certain 
subset of the corresponding
moduli space of closed Riemann surfaces of genus $g$. Occasionally, this 
subset is a complex algebraic curve $C$, then  called {\em Teichm\"uller curve}.
It can be read off from the Veech group, whether  this happens or not. We obtain a Teichm\"uller curve,
if and only if $\Gamma(X^*,\mu)$
has finite volume. In this case one even has that $C$ is birational to $\HH/\Gamma(X^*,\mu)$.
In this article we will just need the following properties of Veech groups:

\begin{itemize}
\item
  The Veech group of a primitive origami is a finite index subgroup of $\slzwei{\ZZ}$ 
  (see \cite[Cor. 2.9]{Sc0} or for a more general discussion \cite{GJ})
\item 
  $\slzwei{\ZZ}$ acts on the set of isomorphism classes of origamis with $d$ squares and 
  the Veech group is the stabiliser of this action.
\item 
  If the translation surface $(X,\mu)$ is defined by an origami, then we have for each
  rational direction $\vec{v} = {p\choose q}$ (with $p,q \in \ZZ$): The surface decomposes in direction
  $\vec{v}$ into finitely many cylinders $C_1$, \ldots, $C_k$. If $m_1$, \ldots, $m_r$ 
  are the inverse moduli of the cylinders, respectively, and $m$ is the smallest common integer
  multiple of the $m_i$'s, then 
  \[ A \cdot \begin{pmatrix} 1& m\\ 0 & 1 \end{pmatrix} A^{-1}\]
  is in the Veech group, where $A$ is an arbitrary matrix in $\slzwei{\ZZ}$
  which maps ${1 \choose 0}$ to $\vec{v}$.
\end{itemize}

We will explain the second and third statement in more detail in the following 
two sections.

\subsection{Action of \boldmath{$\slzwei{\ZZ}$} on origamis}\label{slzweiz-action}

There is a natural action of the group $\slzwei{\RR}$ on 
translation surfaces defined by $A\colon  (X,\mu) \mapsto (X,\mu_A)$,
where $\mu_A$ is the translation structure obtained from 
$\mu$ by composing each chart with the affine map 
\[A\colon {x\choose y} \mapsto A\cdot{x\choose y}.\] 
The stabiliser of $(X,\mu)$ (up to isomorphisms)
is the Veech group $\Gamma(X,\mu)$
(see e.g. \cite[Theorem 1]{EG}).\\
If $(X,\mu)$
comes from a primitive origami $O$ and $A \in \slzwei{\ZZ}$, then
$A\cdot(X,\mu)$ is again the translation structure 
of an origami $O_A = A\cdot O$. 
More precisely,
if $(X,\mu)$ is defined by the covering $p\colon X \to E$,
then $A\cdot(X,\mu) = (X,\mu_A)$ is the origami 
defined by the covering $p_A := A\circ p\colon  
 X \stackrel{p}{\to} E  \stackrel{A}{\to} E$. 
The length of the  orbit of the origami $O$
under this  $\slzwei{\ZZ}$-action on the set of origamis
is the index $[\slzwei{\ZZ}:\Gamma(X,\mu)]$
of the Veech group in $\slzwei{\ZZ}$.
\\
In terms of the corresponding monodromy maps, the action expresses as follows: 
Let $\rho\colon F_2 \to S_d$ be the monodromy map corresponding to $O$,
then $A\cdot\rho = \rho_A = \rho\circ\gamma_A^{-1}$ is the monodromy map
of the origami $A\cdot O$, where $\gamma_A$ is a preimage of $A$
under the natural homomorphism 
$\mbox{Aut}(F_2) \to \mbox{Out}(F_2) \cong \mbox{GL}_2(\ZZ)$.\\
Finally, in terms of the corresponding pair of permutations 
the action of $\slzwei{\ZZ}$ is given as follows: 
The pair $(\sigma_a^{O_A},\sigma_b^{O_A})$ with
\[\sigma_a^{O_A} = \rho(\gamma_A^{-1}(x)) \mbox{ and } 
  \sigma_b^{O_A} = \rho(\gamma_A^{-1}(y))
\]
describes the origami $O_A$. In particular, for the 
matrices 
\begin{equation}\label{matrices}
  T = \bpm 1&1\\0&1\epm, T' = \bpm 1&0\\1&1 \epm \mbox{ and }
  S = \bpm 0&-1\\1&0\epm
\end{equation}
we obtain the permutations:
\[
\begin{array}{lll}
  \sigma_a^{O_T} = \rho(x) = \sigma_a &
  \mbox{ and }& \sigma_b^{O_T} = \rho(x^{-1}y) = \rho(y)\rho(x)^{-1} = \sigma_b\sigma_a^{-1},\\
  \sigma_a^{O_{T'}} = \sigma_b^{-1}\sigma_a & \mbox{ and } &\sigma_b^{O_{T'}} = \sigma_b,\\
  \sigma_a^{O_{S}} = \sigma_b^{-1} & \mbox{ and } & \sigma_b^{O_{S}} = \sigma_a.
\end{array}
\]
\mbox{and furthermore for the inverse matrices:}
\[
\begin{array}{lcll}
  (\sigma_a^{O_{T^{-1}}}, \sigma_b^{O_{T^{-1}}}) &=& (\sigma_a,\sigma_b\sigma_a), & 
  (\sigma_a^{O_{T'^{-1}}}, \sigma_b^{O_{T'^{-1}}}) = (\sigma_b\sigma_a,\sigma_b) \mbox{ and }\\
  (\sigma_a^{O_{S^{-1}}}, \sigma_b^{O_{S^{-1}}})  &=& (\sigma_b,\sigma_a^{-1})&
\end{array}
\]


\subsection{Parabolic elements in the Veech group}\label{cylinders}

Recall that a {\em cylinder} on a translation surface $(X,\mu)$
is a connected set of simple closed geodesics. If $w$ is its
width (also called circumference) and $h$ is its height, then the {\em inverse modulus}
is $m = \frac{w}{h}$. A cylinder $C$ has {\em direction} $\vec{v}$, if
$\vec{v}$ is parallel to the developing vector of the simple closed
geodesics forming $C$. Let us suppose for the moment that 
$\vec{v} =  {1\choose 0}$. Then $C$ is obtained from a rectangle 
of width $w$ and height $h$ by gluing the two vertical edges. The
affine map 
\[
\begin{pmatrix} x\\y \end{pmatrix} \mapsto 
T_m
\cdot \begin{pmatrix}x\\y\end{pmatrix}
\mbox{ with }
T_m =  \begin{pmatrix} 1& m\\ 0&1 \end{pmatrix}
\]
defines an affine map on the cylinder which fixes the boundary pointwise.
Suppose now that the translation surface $(X,\mu)$ decomposes into
horizontal cylinders $C_1$, \ldots, $C_k$ of inverse moduli $m_1$, \ldots, $m_k$.
If m is an integer multiple of each $m_i$, then the map $z \mapsto T_mz$ 
defines on each cylinder an affine map which fixes the boundary pointwise.
These maps glue  to an affine map of the whole surface with derivative $T_m$
which fixes the horizontal saddle connections pointwise. It acts as
the product of powers of Dehn twists along the middle lines of the cylinders.\\
Suppose know that $(X,\mu)$ is an origami given by the permutations $(\sigma_a,\sigma_b)$.
Then the surface naturally decomposes into horizontal (not necessary maximal) 
cylinders of height 1. The number of cylinders is the number of cycles in $\sigma_a$.
The length of a cylinder is the size of the corresponding cycle. In particular,
the matrix $T^n$ is in the Veech group, where $n$ is the least common multiple
of the cycle lengths of $\sigma_a$.\\
For an arbitrary direction $\vec{v}$ we obtain precisely the same
statements if we replace $T$ by its conjugate $ATA^{-1}$ with some arbitrary 
matrix $A$ which maps ${1\choose 0}$ to $\vec{v}$, and the pair of permutations 
$(\sigma_a,\sigma_b)$ by $A^{-1}\cdot (\sigma_A,\sigma_B)$ (see \ref{slzweiz-action}).
In particular we have that 
\[T'_m = \bpm 1 & 0\\ m &1 \epm\] 
is in the Veech group, 
if $m$ is a multiple of the inverse moduli of the vertical cylinders.

\subsection*{Cusps and widths of cusps}
Let us briefly recall some basics from 
the theory of Fuchsian groups that will be important in this article.
The reader not familiar with this can find a good introduction
e.g. in \cite{katok}.
Suppose that $\Gamma$ is a finite index subgroup of $\slzwei{\ZZ}$.
We can build
a fundamental domain of $\Gamma$ by taking finitely many images of our
favourite fundamental domain of $\slzwei{\ZZ}$; each image corresponds to 
one coset of $\Gamma$ in $\slzwei{\ZZ}$. The quotient $C = \HH/\Gamma$
is a Riemann surface of finite type, i.e. it is biholomorphic
to a closed surface $\overline{C}$ with finitely many points called {\em cusps} removed.
The word {\em cusp} is here used in a dual role. A {\em cusp}
of $\Gamma$ (on $\HH$) is a point $x$ on the boundary of $\HH$ which is the fixed
point of some parabolic element $A$ in $\Gamma$. We also call the $\Gamma$-orbits of
such points {\em cusps}, more precisely {\em the cusps of $\HH/\Gamma$}. 
They bijectively correspond to the
points in $\overline{C}\backslash C$.\\
For a fixed cusp $x$ of $\Gamma$ on the boundary of $\HH$ 
the stabiliser of $x$ in $\Gamma$ is a cyclic group generated
by a parabolic element $P$. Since $\Gamma$ is a subgroups of $\slzwei{\ZZ}$,
$P$ is  conjugate to $T^n$ with $n \in \ZZ$ and $T$ the matrix from (\ref{matrices}).
We call the absolute value of $n$ the {\em width} of the cusp $x$.\\
The cusps of $\Gamma$ on $\HH$ bijectively correspond to maximal
cyclic parabolic subgroups of $\Gamma$. Two cusps are in the same $\Gamma$-orbit
if and only if the corresponding maximal parabolic subgroups
are conjugated in $\Gamma$. Accordingly, 
the cusps of $\HH/\Gamma$ correspond to the conjugacy classes of maximal cyclic
parabolic subgroups of $\Gamma$. In particular the widths of cusps
in the same orbit are equal
and we may speak of the {\em width of cusps of \boldmath{$\HH/\emslzwei{\ZZ}$}}.
The width of a cusp is the number of copies of the fundamental domain of $\slzwei{\ZZ}$
which lie around it.


\section{The deficiency of being a non-congruence group}\label{section-deficiency}

Recall that a subgroup $\Gamma$ of SL$(2,\ZZ)$ is called a 
{\em congruence group of level $l$}, if $\Gamma$ contains the principal
congruence group  $\Gamma(l)$ of level $l$, where
\[\Gamma(l) = \{A |\; A \equiv \bpm 1&0\\0&1 \epm \mbox{ mod } l \}.\]
{\em The minimal congruence level} of  $\Gamma$ is then the smallest number
$l$ such that $\Gamma$ is a congruence group of level $l$.
$\Gamma$ is called a {\em congruence group} if it is a congruence group
of level $l$ for some $l$, otherwise it is called a {\em non-congruence group}.\\

In general we have for an arbitrary finite index subgroup $\Gamma$ of SL$(2,\ZZ)$
and for an arbitrary natural number $m$ the following commutative diagram
of exact sequences:
\begin{equation}\label{deficiency-diagram}
  \xymatrix{
    1 \ar[r]  & \Gamma(m) \ar[r] & \slzwei{\ZZ}\ar[r]& \slzwei{\ZZ/m\ZZ}\ar[r] & 1\\
    1 \ar[r]  & \Gamma(m)\cap \Gamma \ar[r]\ar@{^{(}->}[u] & \Gamma \ar[r]\ar@{^{(}->}[u]& 
    \Gammaquer \ar[r]\ar@{^{(}->}[u] & 1\\
  },
\end{equation}
where $\Gammaquer$ is the image of $\Gamma$ in $\slzwei{\ZZ/m\ZZ}$
by the natural projection
\[p_m\colon  \slzwei{\ZZ} \to \slzwei{\ZZ/m\ZZ}.\]
Consider the three indices $d := [\slzwei{\ZZ}:\Gamma]$, 
$e := [\slzwei{\ZZ/m\ZZ}:\Gammaquer]$ and 
$f:= [\Gamma(m):\Gamma(m)\cap\Gamma]$.
We then have from the diagram that $d = e\cdot f$. Observe that $f = 1$
if and only if $\Gamma$ is a congruence group of level $m$. In general, $f$ is a factor of the index 
$d$.

\begin{definition}
Let $\Gamma$ be a finite index subgroup of $\slzwei{\ZZ}$.
The {\em deficiency} $f = f_m = f_{\Gamma,m}$
of $\Gamma$ with respect to $m$ is the index $f = [\Gamma(m):\Gamma(m)\cap \Gamma]$.
\end{definition}

\begin{remark}\label{ftod}
We directly read off from (\ref{deficiency-diagram}) that:
\[
  f_m = \frac{d}{e_m} 
  \mbox{ with } d = [\slzwei{\ZZ}: \Gamma]
  \mbox{ and } e_m = [\slzwei{\ZZ/m\ZZ}:p_m(\Gamma)]
\]
\end{remark}

Recall that for a general finite index subgroup $\Gamma$ of $\slzwei{\ZZ}$
we can define the general Wohlfahrt level $l$ as follows.

\begin{definition} 
Let $\Gamma$ be a finite index subgroup of $\slzwei{\ZZ}$.
The {\em Wohlfahrt level} $l = l(\Gamma)$ of $\Gamma$ is defined by
\[l = l(\Gamma) =  \mbox{lcm}(\{b|\, b \mbox{ is the width of a cusp of $\Gamma$}\}) .\]
We will often call $l$  just {\em the level} of $\Gamma$.
\end{definition}
The Wohlfahrt level generalises the congruence level. This result,
which is  crucial for what we do, was proven by Wohlfahrt in \cite{Wohlfahrt}.

\begin{theoremA}[Wohlfahrt]\label{wohlfahrt}
If $\Gamma$ is a congruence group, then the Wohlfahrt level
equals the congruence level of $\Gamma$. 
\end{theoremA}

In particular this means for a congruence group that the
deficiency $f_l$ equals 1, if $l$ is its (Wohlfahrt) level.  In this
section we will show that for a general finite index subgroup $\Gamma$ 
of $\slzwei{\ZZ}$ the deficiency $f_m$
is minimised, if $m$ equals the Wohlfahrt level $l(\Gamma)$.
Thus let for the rest of this section
$\Gamma$ be a fixed subgroup of $\slzwei{\ZZ}$ of finite index
with Wohlfahrt level $l = l(\Gamma)$.\\

Observe first that we immediately  obtain  from the definition of the Wohlfahrt level
the following remark.

\begin{remark}\label{level-comonotony}
Let $\Gamma$ and $\Gamma '$ be two subgroups of 
$\slzwei{\ZZ}$. If $\Gamma \subseteq \Gamma '$, then $l(\Gamma ')$ divides $l(\Gamma)$.
\end{remark}

We will further need the following property of the 
deficiency.

\begin{lemma}\label{multiple}
For a multiple $k\cdot a$ of a natural number $a$  we have for the
corresponding deficiencies:
\[f_{ka} \mbox{ is a divisor of } f_a.\]
\end{lemma}

\begin{proof}
Consider the natural projection $q\colon  \slzwei{\ZZ/ka\ZZ} \to \slzwei{\ZZ/a\ZZ}$.
Then we have $q(p_{ka}(\Gamma)) = p_{a}(\Gamma)$. Thus we obtain for the indices that
$e_a$ divides $e_{ka}$. Hence the claim follows from Remark~\ref{ftod}.
\end{proof}

We are now ready to prove Theorem~\ref{thm1}, which we stated in
the introduction.

\begin{proof}[Proof of Theorem~\ref{thm1}]
From Lemma~\ref{multiple} we have that $f_m \geq f_{ml}$. 
We now show that $f_{ml} = f_l$. Consider
 the group $\Gamma' = p_{ml}^{-1}(p_{ml}(\Gamma))$. It has 
by its construction the same image in $\slzwei{\ZZ/lm\ZZ}$ 
as $\Gamma$, i.e. $p_{lm}(\Gamma) = p_{lm}(\Gamma')$ and $\Gamma$
is contained in $\Gamma'$.
Furthermore $\Gamma'$ is by its 
definition a congruence group. Let $l'$ be its congruence level
which equals the Wohlfahrt level by Theorem A.
We obviously have $\Gamma(l') \subseteq \Gamma'$.  
Furthermore, by Remark~\ref{level-comonotony} 
we have that $l'$ divides $l = l(\Gamma)$ and thus $\Gamma(l) \subseteq \Gamma(l')$. 
Altogether we obtain that 
$p_{lm}(\Gamma) = p_{lm}(\Gamma') \supseteq p_{lm}(\Gamma(l))$. \\
Consider now the following commutative diagram of exact sequences:
\[
  \xymatrix{
    1 \ar[r]  & p_{lm}(\Gamma(l)) \ar[r] & \slzwei{\ZZ/lm\ZZ}\ar[r]& 
      \slzwei{\ZZ/l\ZZ}\ar[r] & 1\\
    1 \ar[r]  & p_{lm}(\Gamma(l))\cap p_{lm}(\Gamma) \ar[r]\ar@{^{(}->}[u] & 
      p_{lm}(\Gamma) \ar[r]\ar@{^{(}->}[u]& 
      p_{l}(\Gamma) \ar[r]\ar@{^{(}->}[u] & 1}
\]
Since $ p_{lm}(\Gamma(l))$ is contained in $p_{lm}(\Gamma)$,
we have that  $p_{lm}(\Gamma(l))\cap p_{lm}(\Gamma) = p_{lm}(\Gamma(l))$, i.e.
$[p_{lm}(\Gamma(l)):p_{lm}(\Gamma(l))\cap p_{lm}(\Gamma)] = 1$. Thus
we obtain from the diagram 
\[e_{lm} = [\slzwei{\ZZ/lm\ZZ}:p_{lm}(\Gamma)] = 
    1 \cdot [\slzwei{\ZZ/l\ZZ}:p_{l}(\Gamma)] = e_l .\] 
From this
we obtain $f_{lm} = \frac{d}{e_{lm}} = \frac{d}{e_l} =  f_l$ by Remark~\ref{ftod}.
\end{proof}

This result suggests to consider the deficiency 
of $\Gamma$ with respect to the Wohlfahrt level as a measure
of ``how far away'' $\Gamma$ is from being a congruence group. 
The bigger it is, the further the group is away from being a congruence group.
Equivalently, we may consider the number $e = \frac{d}{f}$.
If it is 1, then we are ``as far possible'' from being a congruence group.

\begin{definition}\label{honest-family}
Let $l$ be the Wohlfahrt level of $\Gamma$ and $f = f_l$,
$e = e_l$ as above.
We call $f = f(\Gamma)$ the {\em  deficiency of $\Gamma$ (of being a congruence group)}
and $e = e(\Gamma)$ its {\em level index}. If $e = 1$, then we say that $\Gamma$
is a {\em totally non-congruence group}. For an infinite family $\Gamma_n$,
we say it is an {\em honest family of non-congruence groups},
if $e(\Gamma_n)$ is bounded by a constant.
\end{definition}

Theorem~\ref{enough_parabolics} will give us a criterion
to show that the Veech group 
of an origami is a totally non-congruence group. It will be crucial
that the Wohlfahrt levels of conjugated groups are equal.
One main ingredient will be to consider the width of the cusps $\infty$ and $0$.
This motivates the following definition.

\begin{definition}\label{cuspsizepair}
Let $\Gamma$ be a subgroup of $\slzwei{\ZZ}$ of finite index
which has cusps at $\infty$ and $0$ of width $a$ and $b$, respectively. We call 
the pair $(a,b)$
the {\em normalised cusp-width pair} of $\Gamma$. 
\end{definition}

The main tool for proving Theorem~\ref{enough_parabolics}  
will be the following proposition.

\begin{proposition}\label{thestep}
Let $a$ and $b$ be the widths of the cusps $0$ and $\infty$, respectively, of
a finite index subgroup $\Gamma$ of $\emslzwei{\ZZ}$.
Suppose that we can decompose 
the Wohlfahrt level $l$ of $\Gamma$ as $l = N\cdot M$ 
such that 
\begin{itemize}
\item 
  $N$ and $M$ are relatively prime and
\item 
  $n = \emlcm(a,b)$ divides $N$.
\end{itemize}
Then we have in $\emslzwei{\ZZ/l\ZZ} = \emslzwei{\ZZ/N\ZZ} \times \emslzwei{\ZZ/M\ZZ}$: 
\[p_l(\Gamma) \supseteq  \{I\} \times \emslzwei{\ZZ/M\ZZ}.\]
\end{proposition}

\begin{proof}
Write $N = n'\cdot n$ and $n = a\cdot b' = a'\cdot b$
with $n'$, $a'$ and  $b'$ in $\NN$.
Let us first show that  the element $(I, T)$ of 
$\slzwei{\ZZ/N\ZZ} \times \slzwei{\ZZ/M\ZZ}$ lies in $p_l(\Gamma)$. 
Since $N$ and $M$ are relatively prime, there are $k$ and $k'$ in $\ZZ$
such that $k\cdot N + k'\cdot M = 1$ and thus 
$k\cdot N \equiv 1 \mbox{ mod } M$.\\
Since $a$ is the width of the cusp $\infty$ with respect
to $\Gamma$, we have in particular that
$T^a$ and thus also $T^{a\cdot b' \cdot n' \cdot k}$ is in $\Gamma$.
Thus 
\[
  p_l(T^{a\cdot b' \cdot n' \cdot k})  = 
  \bpm 1 & Nk \\ 0 & 1 \epm = 
  (\bpm 1 & 0 \\ 0 & 1 \epm, \bpm 1 & 1 \\ 0 & 1 \epm) 
  \in \slzwei{\ZZ/N\ZZ} \times \slzwei{\ZZ/M\ZZ}
\]
lies in $p_l(\Gamma)$.\\
Starting with the fact that $T'^{\;b}$ is in $\Gamma$, 
we obtain in the same way that 
\[(\bpm 1 & 0 \\ 0 & 1 \epm, \bpm 1 & 0 \\ 1 & 1 \epm) \]
lies in $p_l(\Gamma)$ and thus we conclude the claim.\\
\end{proof}

\begin{proof}[Proof of Theorem~\ref{enough_parabolics}]
Let $l$ be the common Wohlfahrt level of $\Gamma_1$ and $\Gamma_2$.
Let $n_1 = p_1^{k_1}\cdots p_r^{k_r}$
be the prime factorisation. By the definition of the Wohlfahrt level
$l$ we have that
$n_1$ divides $l$ and thus the prime factorisation
of $l$ can be written as 
\[l = p_1^{m_1} \cdot \ldots \cdot p_r^{m_r}
 \cdot p_{r+1}^{m_{r+1}}\cdot \ldots \cdot p_{r+s}^{m_{r+s}}\]
with $k_1 \leq m_1$, \ldots, $k_r \leq m_r$.
Define $N = p_1^{m_1} \cdot \ldots \cdot p_r^{m_r}$
and $M =  p_{r+1}^{m_{r+1}}\cdot \ldots \cdot p_{r+s}^{m_{r+s}}$. 
In particular, we have $l = M\cdot N$, $N$ and $M$ are relatively prime
and $n_1$ divides $N$. It follows from Proposition~\ref{thestep}
that $p_l(\Gamma_1)$ contains $I \times \slzwei{\ZZ/M\ZZ}$.\\
The same is true for $\Gamma_2$ instead of $\Gamma_1$
since these two groups are conjugated, i.e. $p_l(\Gamma_2)$ 
contains $I \times \slzwei{\ZZ/M\ZZ}$.\\
On the other hand, since $n_1$ and $n_2$ are relatively prime,
it follows that $\gcd(n_2,N) = 1$. Furthermore we have that
$n_2$ divides $l$, as $l$ is the Wohlfahrt level of $\Gamma_2$.
Thus $n_2$ in particular divides $M$. We obtain again from 
Proposition~\ref{thestep} that $p_l(\Gamma_2)$ contains 
$\slzwei{\ZZ/N\ZZ} \times \{I\}$. Again by conjugation 
the same is true for $p_l(\Gamma_1)$.
Altogether, we obtain that both groups $p_l(\Gamma_1)$ 
and $p_l(\Gamma_2)$ contain the full group
$\slzwei{\ZZ/N\ZZ} \times \slzwei{\ZZ/M\ZZ} = \slzwei{\ZZ/l\ZZ}$.
\end{proof}

Motivated by the proof of Theorem~\ref{enough_parabolics}
we introduce the following notation.

\begin{definition}\label{mpde}
Suppose that $n$ and $N$ are two numbers such that $n$ divides $N$.
Let $t$ be the largest divisor of $N$ 
such that $t$ and $n$ have the same prime divisors.
We call $t$ the {\em maximal prime divisor equivalent}
of $n$ in $N$ and denote it by $t = \mpde_N(n)$. Observe
that we in particular have that $n$ divides $t$, that $n$ and $t$ have
the same prime divisors and that $t$ and $N/t$ are relatively prime.
\end{definition}

In Section~\ref{hone} we will further need a statement 
similar to Theorem~\ref{enough_parabolics} with weaker prerequisites, 
which we will get from Lemma~\ref{contains}.

\begin{lemma}\label{contains}
Suppose that we are in the situation of Theorem~\ref{enough_parabolics}
except that in (\ref{relativelyprime})  it is not given
that $n_1$ and $n_2$ are relatively prime. I.e.
we have two conjugated groups $\Gamma_1$ and $\Gamma_2$
with normalised cusp-width pairs  $(a_1,b_1)$ and $(a_2,b_2)$,
respectively, and $n_1 = \emlcm(a_1,b_1)$.
 Define 
$N = \emmpde_l(n_1)$ (see Definition~\ref{mpde}) and $M = l/N$ as in the proof of the theorem. 
Let furthermore $g_1 = \gcd(a_2,N)$ 
and $g_2 = \gcd(b_2,N)$. Then we have: 
\[
  \bpm 1 & g_1\\ 0   & 1 \epm \mbox{ and } 
  \bpm 1 & 0  \\ g_2 & 1 \epm \mbox{ are in } 
  p_l(\Gamma_2).
\]
\end{lemma}

\begin{proof}
By its definition $g_1 \equiv k\cdot a_2$
modulo $N$ for some $k \in \NN$.\\
Thus it follows from
$T^{a_2} \in \Gamma_2$
that $p_l(\Gamma_2)$ contains
\begin{equation}\label{geins}
  \bpm 1 & k\cdot a_2\\ 0   & 1 \epm = 
  (\bpm 1 & g_1\\ 0   & 1 \epm, A_2) \in 
  \slzwei{\ZZ/N\ZZ} \times \slzwei{\ZZ/M\ZZ}
\end{equation}
with some matrix $A_2$ in $\slzwei{\ZZ/M\ZZ}$.
As in the proof of Theorem~\ref{enough_parabolics}
it follows from Proposition~\ref{thestep}
that $p_l(\Gamma_1)$ and thus also $p_l(\Gamma_2)$
contains $\{I\} \times \slzwei{\ZZ/M\ZZ}$.
Hence $p_l(\Gamma_2)$ contains in particular $(I,A_2^{-1}\cdot p_M(T^{g_1}))$. 
From this and (\ref{geins}) we obtain
the first part of the claim. The second part
follows in the same way.
\end{proof}


\section{The $L$-origamis}\label{hone}

In this section we show that the family of all Veech groups of origamis
in $\omegam{2}(2)$ is an honest family of non-congruence groups (compare Definition~\ref{honest-family}). 
We furthermore show for a subclass
of them that their Veech groups are even totally non-congruence groups.\\

Recall that we have an explicit classification of $\slzwei{\RR}$-orbits 
of origamis
in $\omegam{2}(2)$ by Hubert/Leli\`evre (for the case that
the number of squares is prime)
and McMullen (in full generality), see \cite[Thm.1.1]{HL} 
and \cite[Cor. 1.2 and Section 6]{McM1}.
They distinguish the orbits only by the number of squares of the origamis 
and the number of integer Weierstra\ss\ points. 
Recall that in genus 2 the {\em Weierstra\ss\ points} are the six fixed
points of the hyperelliptic involution, which is in the case of an origami
the unique affine homeomorphism with derivative $-I$ of order 2 such that
the quotient by it has genus 0. Being 
an {\em integer Weierstra\ss\ point} means that the point 
in addition is a vertex of one of the squares
of the origami (see Example~\ref{lexample} and Example~\ref{cr}).

\begin{theoremB}[Hubert/Leli\`evre, McMullen]
The set of primitive origamis with $n$ squares form one 
single $\slzwei{\RR}$-orbit,
if $n$ is even or 3. They form
two orbits called $A_n$ and $B_n$, if $n$ is odd and unequal to 3.
An origami is in $A_n$, if it has one integer Weierstra\ss\ point
and in $B_n$, if it has three integer Weierstra\ss\ points.
\end{theoremB}

\begin{example}\label{lexample}
Figure~\ref{lfuenf} shows two origamis with 5 squares. Opposite
edges are glued. Both origamis lie
in $\omegam{2}(2)$. The left one has three integer Weierstra\ss\ points, the right one
only one. Thus the left one lies in the orbit $B_n$ and the right one
in the orbit $A_n$.\\

\begin{center}
\setlength{\unitlength}{1cm}
\begin{picture}(4,3)
\put(0,0){\framebox(1,1){}}
\put(0,1){\framebox(1,1){}}
\put(0,2){\framebox(1,1){}}
\put(1,0){\framebox(1,1){}}
\put(2,0){\framebox(1,1){}}
\put(0,0){\circle*{.2}}
\put(2,0){\circle*{.2}}
\put(0,2){\circle*{.2}}
\put(.5,.5){\circle*{.2}}
\put(2,.5){\circle*{.2}}
\put(.5,2){\circle*{.2}}
\end{picture}
\hspace*{1cm}
\begin{picture}(4,3)
\put(0,0){\framebox(1,1){}}
\put(0,1){\framebox(1,1){}}
\put(1,0){\framebox(1,1){}}
\put(2,0){\framebox(1,1){}}
\put(3,0){\framebox(1,1){}}
\put(0 , 0){\circle*{.2}}
\put(.5,.5){\circle*{.2}}
\put(.5,1.5){\circle*{.2}}
\put(0,1.5){\circle*{.2}}
\put(2.5,0){\circle*{.2}}
\put(2.5,.5){\circle*{.2}}
\end{picture}
\end{center}

\begin{center}
\refstepcounter{diagramm}{\it Figure \arabic{diagramm}}:\label{lfuenf}
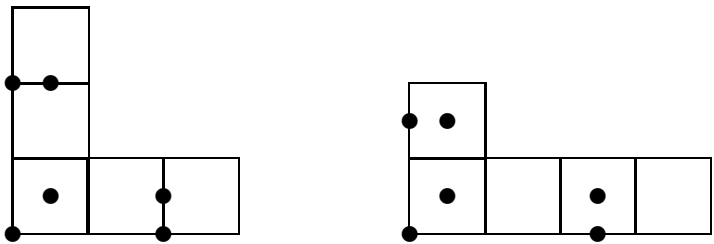
{\em The $L$-shaped origamis $L(3,3)$ and $L(4,2)$ and their Weiterstra\ss\ points. 
      Opposite edges are glued.
     Left one is in $B_n$, right one is in $A_n$.}\\
\end{center}
\end{example}

\begin{remark}\label{lwpts}
One easily observes from Theorem~B that in $\omegam{2}(2)$
each origami-orbit can be represented 
by some $L$-shaped origami $L_{a,b}$ $(a,b \geq 2)$ with $n = a+b-1$
squares, see Figure~\ref{lfuenf}.
If $n$ is odd, then $L_{a,b}$ has 1 integer Weierstra\ss\ point (i.e. belongs
to $A_n$) if $a$ and $b$ are 
even. It has 3 integer Weierstra\ss\ points (i.e. belongs to $B_n$), if $a$ and $b$ are odd.
\end{remark}

One furthermore directly observes from the $L$-shape 
that $\Gamma(L_{a,b})$ contains the parabolic elements
$T^a$ and $T'^{\;b}$ (compare Section~\ref{cylinders}). Moreover, they generate the
corresponding maximal cyclic parabolic subgroup of $\Gamma(L_{a,b})$,
which can be seen as follows.
Any
affine homeomorphism of $L_{a,b}$ whose derivative  
$A \in \Gamma(L_{a,b})$ is parabolic with eigenvector
$1\choose{0}$ has to permute the horizontal saddle connections.
We have three of them. Among them there is a unique one
which lies on the boundary of only one cylinder.
This property is preserved by the affine homeomorphism.
Thus this segment has to be fixed pointwise  and with it
the full boundary of this cylinder.
Hence the affine homeomorphism acts as the power of a Dehn twist 
on this cylinder and $A$ is a power of $T^a$. Hence $T^a$
generates the corresponding cyclic parabolic subgroup of $\Gamma(L_{a,b})$.
We similarly obtain the corresponding statement for $T'^{\;b}$ and with this
the following remark. 

\begin{remark}\label{twocusps}
The widths of $\Gamma(L_{a,b})$ at the cusps $\infty$ and $0$ 
are $a$ and $b$, respectively.
\end{remark}

\begin{example}\label{cr}
Figure~\ref{figurecr} shows an other type of  
origami in $\omegam{2}(2)$. The origami $\Cr_{2,j}$ (with $j$ squares)
is defined by the two permutations
$\sigma_a = (1,2,\ldots, j)$ and $\sigma_b = (1,2)$.
If $j$ is odd, then $\Cr_{2,j}$ has 1 integer Weierstra\ss\ point
and thus belongs to the orbit $A_j$. If $j$ is even, then 
it has 2 integer Weierstra\ss\ points.\\
  
\begin{center}
\setlength{\unitlength}{1cm}
\begin{picture}(7,1)
\put(0,0){\framebox(1,1){1}}
\put(1,0){\framebox(1,1){2}}
\put(2,0){\framebox(1,1){\ldots}}
\put(3,0){\framebox(1,1){}}
\put(4,0){\framebox(1,1){}}
\put(5,0){\framebox(1,1){\ldots}}
\put(6,0){\framebox(1,1){j}}
\put(0,0){\circle*{.2}}
\put(0.5,0){\circle*{.2}}
\put(1,0.5){\circle*{.2}}
\put(1.5,0){\circle*{.2}}
\put(4.5,0){\circle*{.2}}
\put(4.5,.5){\circle*{.2}}
\put(.3,-.4){$a$}
\put(1.3,1.1){$a$}
\put(1.3,-.4){$b$}
\put(.3,1.1){$b$}
\end{picture}
\hspace*{5mm}
\begin{picture}(6,1)
\put(0,0){\framebox(1,1){1}}
\put(1,0){\framebox(1,1){2}}
\put(2,0){\framebox(1,1){\ldots}}
\put(3,0){\framebox(1,1){}}
\put(4,0){\framebox(1,1){\ldots}}
\put(5,0){\framebox(1,1){j}}
\put(0,0){\circle*{.2}}
\put(0.5,0){\circle*{.2}}
\put(1,0.5){\circle*{.2}}
\put(1.5,0){\circle*{.2}}
\put(4,0){\circle*{.2}}
\put(4,.5){\circle*{.2}}
\put(.3,-.4){$a$}
\put(1.3,1.1){$a$}
\put(1.3,-.4){$b$}
\put(.3,1.1){$b$}
\end{picture}
\end{center}
\vspace*{3mm}
\begin{center}
\refstepcounter{diagramm}{\it Figure \arabic{diagramm}}:\label{figurecr}
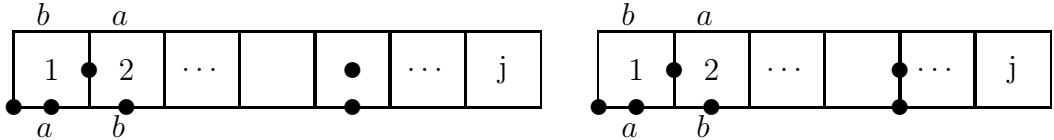
{\em The origamis $\Cr_{2,7}$ and $\Cr_{2,6}$.
  Edges with same labels
and unlabelled opposite edges are identified. The
left one is in $A_7$, the right one has an even number of squares.}\\
\end{center}

One again can directly read off from Figure~\ref{figurecr} that $T^j$ and $T'^2$
are in the Veech group of $\Cr_{2,j}$. Furthermore 
$\Cr_{2,j}$ decomposes into two maximal vertical cylinders,
one of height 1 and circumference 2 and one of height $j-2$ and circumference
1. On the first one we have a vertical saddle connection which lies
on the boundary of only this cylinder. Hence 
a parabolic affine homeomorphism with the vertical direction as eigen direction
has to fix this segment and thus has to act on this cylinder as  multi Dehn twist.
It follows that $T'^{\,2}$ generates the corresponding maximal cyclic parabolic subgroup of
$\Gamma(\Cr_{2,j})$.
If $j \geq 4$, furthermore  
a parabolic affine homeomorphism with the horizontal direction as eigen direction
has to preserve the unique horizontal saddle connection of length $j-2$.  
Hence $T^{j}$ generates the corresponding maximal cyclic parabolic subgroup
of $\Gamma(\Cr_{2,j})$. In particular we have that if $j \geq 4$, then
the normalised pair of cusp-widths for $\Gamma(\Cr_{2,j})$ is $(j,2)$. 
\end{example}

We now want to study  the orbits of origamis in $\omegam{2}(2)$. 
Table~\ref{tabble} shows up to 11 squares for each orbit 
the index of the Veech group $\Gamma$ in $\slzwei{\ZZ}$, 
the genus and number of cusps of $\HH/\Gamma$, 
the level index of $\Gamma$ and its deficiency. Observe that 
these data are
stable under conjugation of $\Gamma$ and therefore do not depend
on the origami but only on the orbit. From the table
one guesses the result of Theorem~\ref{stratumlab} 
which we prove in the following.

\begin{table}
\begin{tabular}{l|r|r|r|r|r|r}
$\sharp$ squares, orbit  &  genus $g$ &    
$\sharp$ cusps $s$ &  level $l$  & index $d$ &  level index $e$ & deficiency $f$\\ 
\hline
   3           &0  &      2  &     2  &     3 &    3 &    1  \\
   4           &0  &      3  &    12  &     9 &    3 &    3  \\
   5,\; $B_5$  &0  &      3  &    15  &     9 &    1 &    9  \\
   5,\; $A_5$  &0  &      5  &    60  &    18 &    3 &    6  \\
   6           &0  &      8  &    60  &    36 &    3 &    12 \\
   7,\; $B_7$  &0  &      8  &   105  &    36 &    1 &    36 \\
   7,\; $A_7$  &0  &     10  &   420  &    54 &    3 &    18 \\
   8           &1  &     17  &   840  &   108 &    3 &    36 \\
   9,\; $B_9$  &0  &     14  &   630  &    81 &    1 &    81 \\
   9,\; $A_9$  &2  &     16  &  2520  &   108 &    3 &    36 \\
  10           &4  &     30  &  2520  &   216 &    3 &    72 \\
  11,\; $B_{11}$&3 &     26  &  6930  &   180 &    1 &   180 \\
  11,\; $A_{11}$&6 &     26  & 27720  &   225 &    3 &    75 \\  
\end{tabular}
\vspace*{3mm}
\caption{For the Veech groups $\Gamma$ of th $L$-origamis up to 11 squares:
  genus $g$ of $\HH/\Gamma$, number $s$ of cusps of $\HH/\Gamma$,
  Wohlfahrt level $l$, index $d$, level index $e$ and deficiency $f$ 
  of $\Gamma$. \protect\footnotemark}
\label{tabble}
\label{proj-index}
\end{table}\footnotetext{calculated with the Software {\em Origami Library},
  \url{http://www.math.kit.edu/iag3/seite/ka-origamis/en}
   }

\begin{proof}[Proof of Theorem~\ref{stratumlab} i)] 
We distinguish the three cases $j \equiv 1$ modulo 4 but not 5, 
$j \equiv 3$ modulo 4 but
not 7 and finally  the singular cases $j = 5$ and $j = 7$.\\[2mm]
{\bf First Step:} Suppose that $j \equiv 1$ modulo $4$ with $j \geq 9$.\\
This means that $j = 2a-1$ with $a$ odd. Furthermore $a \geq 5$.\\
From Hubert/Leli{\`e}vre's and McMullen's classification (see Theorem~B) 
of the $\slzwei{\ZZ}$-orbits of origamis in $\omegam{2}(2)$
and Remark~\ref{lwpts} we know that the origamis $L_{a,a}$ and $L_{a+2, a-2}$ both lie in the orbit
$B_j$. Here we need that $a-2 \geq 2$. 
Then the Veech groups $\Gamma_1 = \Gamma(L_{a,a})$ and
$\Gamma_2 = \Gamma(L_{a+2,a-2})$ are conjugated. Let $l$ be their Wohlfahrt
level. By Remark~\ref{twocusps}
we have that the normalised cusp-width pair of $\Gamma_1$ is $(a,a)$
and that of $\Gamma_2$ is $(a+2,a-2)$. Furthermore $n_1 = \lcm(a,a) = a$
and $n_2 = \lcm(a+2,a-2)$ are relatively prime, since
$a$ is odd. Thus we are in the situation of Theorem~\ref{enough_parabolics}
and obtain that $p_l(\Gamma_1) = p_l(\Gamma_2)$ is the full group 
$\slzwei{\ZZ/l\ZZ}$.\\

{\bf Second Step:} Suppose that $j \equiv 3$ modulo $4$ with $j \geq 11$.\\
Thus $j = 2a-1$ with $a$ even and $a \geq 6$.\\
Again we obtain from Hubert/Leli\`{e}vre's and McMullen's classification 
that in this case $L_{a+1,a-1}$ and $L_{a+3, a-3}$ both lie in the orbit
$B_j$. We use the same arguments as before: We consider the Veech groups
$\Gamma_1 = \Gamma(L_{a+1,a-1})$ and
$\Gamma_2 = \Gamma(L_{a+3,a-3})$. They are conjugated and the 
normalised cusp-width pairs are $(a+1,a-1)$ and $(a+3,a-3)$,
respectively. In this case we have $n_1 = \lcm(a+1,a-1)$
and $n_2 = \lcm(a+3,a-3)$.  Observe that $n_1$ and $n_2$
are relatively prime, since
$a$ is even. Thus we are again in the situation of 
Theorem~\ref{enough_parabolics}
and obtain the desired statement.\\

{\bf Third Step:} The remaining cases $j = 7$ and $j=9$ have been
explicitly calculated (see Table~\ref{tabble}).
\end{proof}

We now show the second part of Theorem~\ref{stratumlab}, i.e.
that for the other types of orbits, namely the case $j$ even, the case
$A_j$ with $j$ odd and the special case $j = 3$,
the projection of the Veech group into $\slzwei{\ZZ/l\ZZ}$ 
is not surjective, but it is ''almost surjective''.
More precisely the index of the image is constantly equal to 3.

\begin{proof}[Proof of Theorem~\ref{stratumlab} ii)]
We first show that the level index is at most 3.
We distinguish the following cases:
\begin{itemize}
\item[a)] 
  $j$ is odd, $j \equiv 3$ modulo 4, $j \geq 7$  and the origami lies in $A_j$,
\item[b)]
  $j$ is odd, $j \equiv 1$ modulo 4, $j \geq 9$ and the origami lies in $A_j$,
\item[c)]
  $j$ is even and  $j \geq 6$
\item[d)]
  $j = 3$, $j = 4$ or $j = 5$ and the orbit is $A_5$.  
\end{itemize}

a) Suppose that $j = 2a-1 \geq 7$ with $a$ even, 
then $L_{a,a}$ 
and $\Cr_{2,j}$ lie in the same orbit (see Theorem~B, Remark~\ref{lwpts},
Example~\ref{cr}). The normalised cusp-width pairs of their Veech groups 
are $(a,a)$  (see Remark~\ref{twocusps})
and $(j,2)$ (see Example~\ref{cr}; here we need that $j \neq 3$). Define $N = \mpde_l(a)$ 
(see Definition~\ref{mpde}).
Then $N$ and $a$ have the same 
prime divisors and thus $2 = \gcd(2,a) = \gcd(2,N)$. 
Furthermore $j = 2a-1$ and $a$
are relatively prime, thus $1 = \gcd(j,a) = \gcd(j,N)$. 
It follows from Lemma~\ref{contains} that 
\[
  p_l(T) = \bpm 1&1\\0&1\epm \mbox{ and } 
  p_l(T'^{\,2}) = \bpm 1&0\\2&1\epm \mbox{ lie in } p_l(\Gamma(Cr_{2,j})).\]
Since furthermore $-I \in \Gamma(Cr_{2,j})$, we have that the image $p_l(\Gamma(2))$ in $\slzwei{\ZZ/l\ZZ}$ of the principal
congruence group $\Gamma(2)$  is contained in $p_l(\Gamma(Cr_{2,j}))$.
Thus its index in $\slzwei{\ZZ/l\ZZ}$ is at most 6. 
Since in addition $p_l(T)$ lies in it, 
the index is at most 3 and we have obtained the desired statement.\\

b) Let $j = 2a-1$ with $a$ odd, $a \geq 5$. Then $L_{a-1,a+1}$, $L_{a-3,a+3}$ 
and $Cr_{2,j}$ are in the same orbit.
We have that $2 = \gcd(2,a-1) = \gcd(2,a+1)$ and thus $2 = \gcd(2,N)$, where
$N = \mpde_l(\lcm(a-1,a+1))$. It follows again by Lemma~\ref{contains} that 
$p_l(T'^{\,2}) \in p_l(\Gamma(Cr_{2,j}))$. \\
Furthermore since $j = 2a-1 = 2(a-1) + 1 = 2(a+1)-3$,
we  have that $\gcd(a-1,j) = 1$ and $\gcd(a+1,j)$ divides 3.
Thus $\gcd(N,j)$ is a power $3^k$ of 3. We obtain from
Lemma~\ref{contains} that $p_l(T^{3^k}) \in p_l(\Gamma(Cr_{2,j}))$.\\
Let us now work with $L_{a-3, a+3}$ instead of 
$L_{a-1,a+1}$. We have $\gcd(a-3,j) = \gcd(a-3,2(a-3)+5)$ divides 5 and 
$\gcd(a+3,j) = \gcd(a+3,2(a+3)-7)$ divides 7. Hence $\gcd(\mpde_l(\lcm(a-3,a+3)),j)$
is of the form $5^{k_1}\cdot 7^{k_2}$ and thus by Lemma~\ref{contains} 
we have that $p_l(\Gamma(Cr_{2,j}))$ contains $p_l(T^{5^{k_1}\cdot 7^{k_2}})$. Since
$3^k$ and $5^{k_1}\cdot 7^{k_2}$ are relatively prime, we obtain
that $p_l(T)$ is in $p_l(\Gamma(Cr_{2,j}))$.
We conclude the claim as in the previous case.\\

c) Suppose  now $j = 2a$ with $a \geq 3$.
Then $L_{a,a+1}$, $L_{a-1,a+2}$ and $L_{2,2a-1}$ are in the same orbit. Let $N = \mpde_l(\lcm(a,a+1))$.
We have  that $\gcd(2,N) = 2$. It follows by Lemma~\ref{contains} that
$p_l(T^2)$ lies in $p_l(\Gamma(L_{2,2a-1}))$.\\ 
Furthermore 
$\gcd(2a-1,a) = 1$ and $\gcd(2a-1,a+1) = \gcd(2(a+1)-3,a+1)$ divides 3, thus we
obtain from Lemma~\ref{contains} that $p_l(T'^{\,3^{k_1}})$ is 
in $p_l(\Gamma(L_{2,2a-1}))$ for some $k_1 \in \NN_0$. Let us now work with $L_{a-1,a+2}$.
We have that $\gcd(2a-1,a-1) = \gcd (2(a-1)+1,a-1) = 1$
and $\gcd(2a-1,a+2) = \gcd(2(a+2)-5,a+2)$ divides 5.\\
Thus we obtain by Lemma~\ref{contains}
that $p_l(T'^{5^{k_2}})$ is 
in $p_l(\Gamma(L_{2,2a-1}))$ for some $k_2 \in \NN_0$. Again we use that $3^{k_1}$ and $5^{k_2}$
are relatively prime in order to conclude that $p_l(\Gamma(L_{2,2a-1}))$ contains
$p_l(T')$. Similarly as in the other cases we conclude the claim.\\

d) These singular cases have been explicitly calculated, see Table~\ref{tabble}.
\\

We now show that the level index is at least 3. More precisely,
we show that the Veech group is contained in a subgroup of $\slzwei{\ZZ}$ whose
image in $\slzwei{\ZZ/2\ZZ}$ is of index 3. Since the level $l$ is 
divisible by 2, we obtain the desired property.\\
Recall that in the 
cases which we consider  the number of integer Weierstra\ss\
points of the origami $O$ is 1 or 2. The Weierstra\ss\ points are preimages
of the $2$-division points $\infty$, $A$, $B$ and $C$ on the
elliptic curve $E= \CC/(\ZZ\oplus \ZZ i)$ with marked point 
$\infty = (\bar{0},\bar{0})$ (see Figure~\ref{divisionpoints}).

\begin{center}
\setlength{\unitlength}{2cm}
\begin{picture}(2,2)
\put(1,1){\framebox(1,1){}}
\put(1,1){\circle*{0.1}}
\put(1.5,1){\circle*{0.1}}
\put(1,1.5){\circle*{0.1}}
\put(1.5,1.5){\circle*{0.1}}
\put(.7,.7){$\infty$}
\put(1.3,.7){$A$}
\put(1.3,1.3){$B$}
\put(.7,1.3){$C$}
\end{picture}
\end{center}
\begin{center}
\vspace*{-15mm}
\refstepcounter{diagramm}{\it Figure \arabic{diagramm}}:\label{divisionpoints}
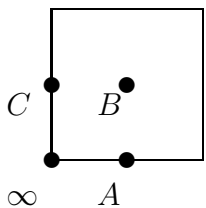
{\em The four $2$-division points of the
  elliptic curve $E = \CC/(\ZZ\oplus\ZZ i)$ 
  with marked point $\infty$.}\\
\end{center}

The integer Weierstra\ss\ points are those 
which are mapped to $\infty$. Since $O$ has 1 or 2 integer
Weierstra\ss\ points, we have 5 or 4 Weierstra\ss\ points which
are preimages of $A$,$B$ or $C$. In particular, one out of the three
points $A$, $B$ and $C$  can be distinguished from the other two by 
the number of Weierstra\ss\
points in its preimage. Since affine maps preserve Weierstra\ss\ points
and each affine map of $X$ descends to $E$, the descend
of an affine map on $X$ has to fix this point. Observe
that for $P \in \{A,B,C\}$
the subgroup of $\slzwei{\ZZ}$ of elements which fix $P$ is 
a congruence group of level 2 and has index $3$ in $\slzwei{\ZZ}$.
This proves the claim.
\end{proof}

\section{One Zero strata}

Theorem~{\ref{enough_parabolics}} and the methods used in Section~\ref{hone}
are quite general. We exemplary use them to construct infinite
families of origamis in each genus $g \geq 3$ whose Veech groups are totally
non-congruence groups and prove with this Theorem~\ref{onezero}. 
More precisely, we introduce an explicit family of origamis $O_{g,n}$ ($g \geq 3$,
$n \geq 3g-2$) whose Veech groups $\Gamma(O_{g,n})$
are totally non-congruence groups if $n$ is prime to  $2$,$3$ and $g-1$.

\begin{definition}\label{defogn}
Let $O_{g,n}$ for $n \geq 3g-2$ and $g \geq 3$ be the origami given by the 
permutations
\[
  \begin{array}{lcl}
    \sigma_a &:=& (1,4,7,\ldots,3g-5, 3g-2,3g-1,\ldots, n) \\
    \sigma_b &:=& (1,2,3)(4,5,6)\ldots(3g-5,3g-4,3g-3).
  \end{array}
\]
The origami is shown in Figure~\ref{ogn}. For easier calculations
later, we write down $\sigma_a$ and $\sigma_b$ also as functions:
\begin{equation}\label{functions}
  \begin{array}{lcl}
    \sigma_a(x) &=&
    \left\{
    \begin{array}{ll}
      x+3, &x \equiv 1 \mod 3 \mbox{ and } x \leq 3g-5\\
      x,   &x \not\equiv 1 \mod 3 \mbox{ and } x \leq 3g-3\\
      x+1, &3g-2 \leq x \leq n-1\\
      1,   &x=n
    \end{array}
    \right.\\\vspace*{2mm}
    \sigma_b(x) &=&
    \left\{
    \begin{array}{ll}
      x+1, &x \not\equiv 0 \mod 3 \mbox{ and } x \leq 3g-3\\
      x-2,   &x \equiv 0 \mod 3 \mbox{ and } x \leq 3g-3\\
      x, & x \geq 3g-2
  \end{array}
    \right.
  \end{array}
\end{equation}
\end{definition}

\vspace*{5mm}

\begin{center}
  \begin{figure}
  \setlength{\unitlength}{1cm}
  \begin{picture}(8,5)
    \put(0,2){\framebox(1,1){1}}
    \put(0,3){\framebox(1,1){2}}
    \put(0,4){\framebox(1,1){3}}
    \put(1,2,){\framebox(1,1){4}}
    \put(1,0){\framebox(1,1){5}}
    \put(1,1){\framebox(1,1){6}}
    \put(2,2,){\framebox(1,1){7}}
    \put(2,3){\framebox(1,1){8}}
    \put(2,4){\framebox(1,1){9}}
    \put(3,2){\framebox(1,1){\ldots}}
    \put(3,0){\framebox(1,1){}}
    \put(3,1){\framebox(1,1){}}
    \put(4,2){\framebox(1,1){3g-5}}
    \put(4,3){\framebox(1,1){3g-4}}
    \put(4,4){\framebox(1,1){3g-3}}
    \put(5,2){\framebox(1,1){3g-2}}
    \put(6,2){\framebox(1,1){3g-1}}
    \put(7,2){\framebox(1,1){\ldots}}
    \put(8,2){\framebox(1,1){n}}
    \put(0,2){\circle*{.2}}
    \put(0,3){\circle*{.2}}
    \put(0,5){\circle*{.2}}
    \put(1,2){\circle*{.2}}
    \put(1,3){\circle*{.2}}
    \put(1,5){\circle*{.2}}
    \put(1,0){\circle*{.2}}
    \put(2,0){\circle*{.2}}
    \put(2,2){\circle*{.2}}
    \put(2,3){\circle*{.2}}
    \put(2,5){\circle*{.2}}
    \put(3,0){\circle*{.2}}
    \put(3,2){\circle*{.2}}
    \put(3,3){\circle*{.2}}
    \put(3,5){\circle*{.2}}
    \put(4,0){\circle*{.2}}
    \put(4,2){\circle*{.2}}
    \put(4,3){\circle*{.2}}
    \put(4,5){\circle*{.2}}
    \put(5,2){\circle*{.2}}
    \put(5,3){\circle*{.2}}
    \put(5,5){\circle*{.2}}
    \put(9,2){\circle*{.2}}
    \put(9,3){\circle*{.2}}
  \end{picture}\\[3mm]
\refstepcounter{diagramm}{\it Figure \arabic{diagramm}}:\label{ogn}
{\em The origami $O_{g,n}$. Opposite edges are glued}
\end{figure}
\end{center}

We directly read off from Figure~\ref{ogn}
that $O_{g,n}$ has 1 singularity. The corresponding
vertices are shown as marked points in the figure.
In addition we have  
$n - 2g + 1$ regular vertices.
From the Euler characteristic formula
we obtain the genus. One further sees from
Figure~\ref{ogn} that the saddle connections
span $\ZZ^2$.

\begin{remark}\label{cuspsizepairs1}
The origami $O_{g,n}$ is primitive, has genus $g$ and 
lies in the stratum $\omegam{g}(2g-2)$. It has $g$ maximal
horizontal
cylinders: one of circumference $n-2g+2$ and height $1$
and $g-1$ cylinders of circumference $1$ and height $2$.
It further has $g$ maximal vertical cylinders:
$g-1$ cylinders with circumference $3$ and height $1$ and one
cylinder with circumference $1$ and height $n-(3g-3)$.
It follows that the Veech group $\Gamma(O_{g,n})$ contains 
\[\bpm 1&n-2g+2\\ 0&1 \epm \mbox{ and } \bpm 1&0\\3&1 \epm\]
and the normalised cusp-width pair of $O_{g,n}$ is $(a,b)$ with 
$a$ divides $n-2g+2$ and $b$ divides $3$.  
\end{remark}

We now show that there is an origami $A\cdot O_{g,n}$ ($A \in \slzwei{\ZZ}$)
in the same orbit which has only one cylinder 
with respect to the horizontal as well as with respect to the vertical direction.
Recall from Section~\ref{slzweiz-action} that if an origami $O$ is given by a permutation 
$(\sigma_a,\sigma_b)$, then $T^{-1}\cdot O$  and $T'^{-1}\cdot O$
are given by
the pairs 
\begin{equation}\label{trafos}
(\sigma_a,\sigma_b\sigma_a) \mbox{ and }  
(\sigma_b\sigma_a,\sigma_b) \mbox{ , respectively.}
\end{equation}
Applying these two transformations  and their inverse
consecutively, we obtain the origamis in the orbit of $O$.

\begin{lemma}
Let $A = T^{-1}{T'}^{-1}$ and let $O_{g,n}$ be the origami from Definition~\ref{defogn}. 
If $g \geq 3$ and $n \geq 3g-2$, then for the origami $A\cdot O_{g,n}$ the horizontal and
the vertical directions are both one-cylinder directions.
\end{lemma}

\begin{proof}
Let $(\sigma_a, \sigma_b)$ be the pair of permutations
from Definition~\ref{defogn}  describing $O_{g,n}$.
By (\ref{trafos}) we have that $A\cdot O_{g,n}$ is given by the
permutations $(\sigma_b \circ \sigma_a, \sigma_b^2 \circ \sigma_a)$.
We now calculate these permutations. From~(\ref{functions})
we obtain:
\[
    \sigma_b(\sigma_a(x)) =
    \left\{
    \begin{array}{ll}
      x+4, &x \equiv 1 \mod 3 \mbox{ and } x \leq 3g-8\\
      x+1,   &x \equiv 2 \mod 3 \mbox{ and } x \leq 3g-4\\
      x-2,   &x \equiv 0 \mod 3 \mbox{ and } x \leq 3g-3\\
      3g-2 = x+3, &x = 3g-5\\
      x+1, &3g-2 \leq x \leq n-1\\  
      2,   &x=n
    \end{array}
    \right.\\\vspace*{2mm}
\]
Written as a permutation this is:
\[
\begin{array}{lcl}
  \sigma_b\sigma_a &=& (1,5,6,4,8,9,7,11,12,\ldots, 3g-8,3g-4,3g-3,\\
    && \phantom{(\;} 3g-5,3g-2,3g-1,3g,3g+1, \ldots n-1,n,2,3)
\end{array}
\]
In particular, this is an $n$-cycle. Furthermore, we have:
\[
\sigma_b^2(x) = 
\left\{
\begin{array}{ll}
  x+2, &x \equiv 1 \mod 3 \mbox{ and } x \leq 3g-3\\
  x-1,   &x \not\equiv 1 \mod 3 \mbox{ and } x \leq 3g-3\\
  x, & x \geq 3g-2
\end{array}
\right.
\]
We then get:
\[
\sigma_b^2(\sigma_a(x)) =
\left\{
\begin{array}{ll}
  x+5, &x \equiv 1 \mod 3 \mbox{ and } x \leq 3g-8\\
  x-1,   &x \not\equiv 1 \mod 3 \mbox{ and } x \leq 3g-3\\
  3g-2 = x+3, &x = 3g-5\\
  x+1, &3g-2 \leq x \leq n-1\\  
  3,   &x=n
\end{array}
\right.\\\vspace*{2mm}
\]
Written as a permutation this is:
\[
\begin{array}{lcl}
  \sigma_b^2\sigma_a &=& (1,6,5,4,9,8,\ldots,3g-8,3g-3,3g-4,\\
     && \phantom{(\;} 3g-5,3g-2,3g-1,3g, \ldots, n,3,2)
\end{array}
\]
In particular this is again an $n$-cycle.
This finishes the proof.
\end{proof}

We immediately obtain the following corollary.

\begin{cor}\label{cuspsizepairs2}
The Veech group $\Gamma(A\cdot O_{g,n})$ contains $T^n$
and $T'^{n}$. Thus for its normalised cusp-width pair $(a',b')$
we have that $a'$ and $b'$ divide $n$.
\end{cor}

We now can conclude Theorem~\ref{onezero} 
from the following proposition.

\begin{proposition}
If $n$ is coprime to $3$ and coprime to $2g-2$, then
the Veech group $\Gamma(O_{g,n})$ is a totally non
congruence group.
\end{proposition}

\begin{proof}
Recall that $\Gamma(O_{g,n})$ and $\Gamma(A\cdot O_{g,n})$
are conjugated. More precisely we have that $\Gamma(A\cdot O_{g,n})  
= A \Gamma(O_{g,n}) A^{-1}$. We can
apply Theorem~\ref{enough_parabolics} for $\Gamma_1 = \Gamma(O_{g,n})$ 
and $\Gamma_2 = \Gamma(A\cdot O_{g,n})$, since by Remark~\ref{cuspsizepairs1} 
and Corollary~\ref{cuspsizepairs2} the least common multiple of the
normalised cusp-width pair 
of $\Gamma_1$ divides  $\lcm(n-2g+2,3)$ and the least common multiple of the
normalised cusp-width pair of $\Gamma_2$ divides $n$. By the assumptions 
of this proposition it follows that they are coprime.
\end{proof}

\bibliographystyle{amsalpha}
\bibliography{cg}

\end{document}